\documentclass[10pt]{amsart}

\usepackage[T2A]{fontenc}
\usepackage{amsmath}
\usepackage{amsfonts}
\usepackage{amssymb}
\usepackage{amsthm}
\usepackage[toc,page]{appendix}
\usepackage{orcidlink}
\usepackage{amscd}

\usepackage[ukrainian,english]{babel}
\theoremstyle{plain}

\theoremstyle{definition}

\theoremstyle{plain}
\newtheorem*{theoremUkr*}{Theorem}
\newtheorem*{lemmaUkr*}{Lemma}
\newtheorem*{propositionUkr*}{Proposition}
\newtheorem*{statementUkr*}{Statement}
\newtheorem*{corollaryUkr*}{Corollary}
\theoremstyle{definition}
\newtheorem*{definitionUkr*}{Definition}
\theoremstyle{remark}
\newtheorem*{notationUkr*}{Notation}
\newtheorem*{remarkUkr*}{Remark}
\theoremstyle{plain}
\newtheorem{theoremEng}{Theorem}
\newtheorem{lemmaEng}[theoremEng]{Lemma}
\newtheorem{propositionEng}[theoremEng]{Proposition}

\newtheorem{theorem}{Theorem}

\newtheorem{corollary}{Corollary}
\theoremstyle{definition}

\let\C\relax
\DeclareRobustCommand{\C}{\ensuremath{\mathbb{C}}}
\newcommand{\R}{\mathbb{R}}
\newcommand{\N}{\mathbb{N}}

\DeclareMathOperator{\Dom}{Dom}
\DeclareMathOperator{\Ran}{Ran}
\DeclareMathOperator{\Ker}{Ker}
\DeclareMathOperator{\supp}{supp}
\DeclareMathOperator{\AC}{AC}

\usepackage[mathlines]{lineno}
\newcommand*\patchAmsMathEnvironmentForLineno[1]{%
  \expandafter\let\csname old#1\expandafter\endcsname\csname #1\endcsname
  \expandafter\let\csname oldend#1\expandafter\endcsname\csname end#1\endcsname
  \renewenvironment{#1}%
     {\linenomath\csname old#1\endcsname}%
     {\csname oldend#1\endcsname\endlinenomath}%
}
\newcommand*\patchBothAmsMathEnvironmentForLineno[1]{%
  \patchAmsMathEnvironmentForLineno{#1}%
  \patchAmsMathEnvironmentForLineno{#1*}%
}
\AtBeginDocument{%
  \patchBothAmsMathEnvironmentForLineno{equation}%
  \patchBothAmsMathEnvironmentForLineno{align}%
  \patchBothAmsMathEnvironmentForLineno{flalign}%
  \patchBothAmsMathEnvironmentForLineno{alignat}%
  \patchBothAmsMathEnvironmentForLineno{gather}%
  \patchBothAmsMathEnvironmentForLineno{multline}%
}

\usepackage{xcolor}
\usepackage{hyperref}
\hypersetup{
    colorlinks=true,
    linkcolor={blue!60!black},  % Темно-синій для посилань
    citecolor={blue!50!black},  % Трохи інший відтінок для цитат
    urlcolor={blue!80!black}    % Яскравіший синій для веб-сайтів
}

\title[Semigroup properties of Schr\"{o}dinger operators]{Semigroup properties of Schr\"{o}dinger operators  with  complex singular potentials}

\author[V. Mikhailets]{Vladimir Mikhailets}

\address{Vladimir Mikhailets
         \\ \orcidlinkf{0000-0002-1332-1562}
         }

\address{Department of Mathematics \\
         King's College London \\
         2CR2LS \\
         United Kingdom}

\address{Institute of Mathematics of NAS of Ukraine \\
         3, Tereschenkivska Str. \\
         Kyiv-4 \\
         Ukraine \\
         01024}

\email[Vladimir Mikhailets]{mikhailets@imath.kiev.ua, vladimir.mikhailets@gmail.com}

\author[V. Molyboga]{Volodymyr Molyboga}

\address{Volodymyr Molyboga
        \\ \orcidlinkf{0000-0002-5683-5694}
         }

\address{Institute of Mathematics of NAS of Ukraine \\
         3, Tereschenkivska Str. \\
         Kyiv-4 \\
         Ukraine \\
         01024}

\email[Volodymyr Molyboga]{molyboga@imath.kiev.ua, vm.imath@gmail.com}



\begin{document}

%\linenumbers % Вмикає нумерацію рядків для всього документа
%\modulolinenumbers[5]

%%%%%%%%%%%%%%%%%%%%%%%%%%%%%%%%%%%%%%%%%%%%%%%%%%%%%%%%%%%%%%%%%%%%%%%%%%%%%%%%%%%%%%%%%%%%%%%%%%%%%%%%%%%%%%%%%%%%%%%%%%%%
%%%%%%%%%%%%%%%%%%%%%%%%%%%%%%%%%%%%%%%%%%%%%%%%%%%%%%%%%%%%%%%%%%%%%%%%%%%%%%%%%%%%%%%%%%%%%%%%%%%%%%%%%%%%%%%%%%%%%%%%%%%%
\begin{abstract}
In the complex Hilbert space $L_2(\R_+)$ on the half-line $\R_+ = [0, \infty)$, we investigate non-symmetric Schr\"odinger operators
\[
Su = -u'' + qu, \qquad u(0) = 0,
\]
with complex-valued, highly singular distributional potentials
\[
q = s + Q' \in H_{\mathrm{loc}}^{-1}(\R_+), \qquad s \in L_{1,\mathrm{loc}}(\R_+), \quad Q \in L_{2,\mathrm{loc}}(\R_+),
\]
where the derivative $Q'$ is understood in the sense of Schwartz distributions. A rigorous operator-theoretic setting is established within the framework of Shin--Zettl quasi-derivatives. We prove that every accretive operator $S$ of this class is automatically $m$-accretive in $L_2(\R_+)$ and $C$-self-adjoint with respect to the complex conjugation involution $C : u \mapsto \bar{u}$. In particular, no additional boundary conditions at infinity are needed, meaning that a Weyl limit-point type situation always prevails for accretive realizations. Constructive local conditions on the coefficients $s$ and $Q$  are found under which the generated strongly continuous semigroup $\{e^{-tS}\}_{t \ge 0}$ is contractive or holomorphic in a sector of the complex plane.

\end{abstract}

%%%%%%%%%%%%%%%%%%%%%%%%%%%%%%%%%%%%%%%%%%%%%%%%%%%%%%%%%%%%%%%%%%%%%%%%%%%%%%%%%%%%%%%%%%%%%%%%%%%%%%%%%%%%%%%%%%%%%%%%%%%%
%%%%%%%%%%%%%%%%%%%%%%%%%%%%%%%%%%%%%%%%%%%%%%%%%%%%%%%%%%%%%%%%%%%%%%%%%%%%%%%%%%%%%%%%%%%%%%%%%%%%%%%%%%%%%%%%%%%%%%%%%%%%

\maketitle

%%%%%%%%%%%%%%%%%%%%%%%%%%%%%%%%%%%%%%%%%%%%%%%%%%%%%%%%%%%%%%%%%%%%%%%%%%%%%%%%%%%%%%%%%%%%%%%%%%%%%%%%%%%%%%%%%%%%%%%%%%%%
%%%%%%%%%%%%%%%%%%%%%%%%%%%%%%%%%%%%%%%%%%%%%%%%%%%%%%%%%%%%%%%%%%%%%%%%%%%%%%%%%%%%%%%%%%%%%%%%%%%%%%%%%%%%%%%%%%%%%%%%%%%

\section{Introduction}

Let $H$ be a complex Hilbert space with inner product $(\cdot, \cdot)_H$, and let $T$ be a linear operator in $H$ with numerical range 
\[
\Theta(T) := \left\{ (Tx, x)_H \in \C \;\middle|\; \|x\|_H = 1, \; x \in \Dom(T) \right\}.
\]
An operator $T$ is called \emph{accretive} if $\Theta(T) \subset \{\lambda \in \C \mid \mathrm{Re}\,\lambda \ge 0\}$, and \emph{sectorial} with respect to the sector $\Lambda_\delta := \{\lambda \in \C \mid |\arg \lambda| \le \frac{\pi}{2} - \delta\}$, $\delta > 0$, if $\Theta(T) \subset \Lambda_\delta$. 

An operator $T$ is called \emph{maximally accretive} if it admits no proper accretive extensions in $H$, and \emph{$m$-accretive} if the open left half-plane $\{\lambda \in \C \mid \mathrm{Re}\,\lambda < 0\}$ is entirely contained in the resolvent set $\rho(T)$. For densely defined or closable operators, these two concepts are equivalent \cite{Phl1959}. By the celebrated Phillips generation theorem \cite{Phl1959} and the classical Lumer--Phillips theorem, the $m$-accretivity of a densely defined operator $T$ is a necessary and sufficient condition for $-T$ to be the infinitesimal generator of a strongly continuous semigroup of contractions $\{e^{-tT}\}_{t \ge 0}$:
\[
\|e^{-tT}\|_H \le 1, \qquad t \ge 0.
\]
If, in addition, the numerical range of the operator is contained in a sector $\Lambda_\delta$, then the operator $T$ is called \emph{$m$-sectorial}, and the semigroup $\{e^{-tT}\}_{t \ge 0}$ generated by $-T$ admits an analytic continuation to a sector containing the positive half-axis, i.e., it is a \emph{holomorphic semigroup} (see, e.g., \cite{Kato1995}).

In various problems of modern mathematical physics---such as quantum mechanics with absorption, wave propagation in highly heterogeneous or dissipative media, and diffusion theory---differential operators with complex singular coefficients arise naturally. As a rule, such singular coefficients can be interpreted as Schwartz distributions belonging to certain Sobolev spaces. The spectral and semigroup analysis of such operators is notoriously delicate due to fundamental mathematical difficulties related to the definition of products of distributions and the loss of standard smoothness of functions in their domains.

In the present paper, we investigate non-self-adjoint Schr\"odinger operators in the complex Hilbert space $L_2(\R_+)$, where $\R_+ := [0, \infty)$, associated with the formal differential expression
\begin{equation*}\label{eq:1.1}
Su := l[u], \qquad l[u] := -u'' + qu, \qquad x \in (0, \infty),
\end{equation*}
subject to the homogeneous Dirichlet boundary condition at the origin:
\[
u(0) = 0.
\]
The potential $q$ is assumed to be complex-valued and to belong to the negative Sobolev class $H_{\mathrm{loc}}^{-1}(\R_+)$, admitting the representation
\begin{equation}\label{eq:1.2}
q = s + Q', \qquad s \in L_{1,\mathrm{loc}}(\R_+), \quad Q \in L_{2,\mathrm{loc}}(\R_+),
\end{equation}
where the derivative $Q'$ is understood in the distributional sense, and the locality condition means that $s \in L_1[0, a]$ and $Q \in L_2[0, a]$ for every $a > 0$. 

In particular, if the function $Q$ has bounded variation on every compact interval of $\R_+$ ($Q \in BV_{\mathrm{loc}}(\R_+)$), then the potential $q$ is a complex Radon measure. If the singular continuous component of this measure vanishes, one arrives at the widely studied class of point interactions, see, for instance, \cite{AlbGszHKH2005, AlKsMl2010, Brsch1985} and references therein.

The paper is structured as follows.

In Section~\ref{sec:Prlm}, we gather preliminary results regarding operators associated in $L_2(\R_+)$ with the formal expression $l$ and its formal Lagrange adjoint $l^+$. The cornerstone of this section is Theorem~\ref{thm:2.1}, which establishes the essential structural properties of maximal, minimal, and preminimal operators, their adjoint relations, domain embeddings into local Sobolev spaces, and their $C$-symmetry properties.

Section~\ref{sec:MnRslts} contains the statements of the main theorems: Theorem~~\ref{thm:3.1} provides constructive local conditions for accretivity and sectoriality; Theorem~\ref{thm:3.2} establishes the global uniqueness and $m$-accretivity theorem; Theorem~\ref{thm:3.3} describes the $C$-self-adjointness of the minimal operator.

Section~\ref{sec:thm3.1} is devoted to the proof of Theorem~\ref{thm:3.1}.

Section~\ref{sec:thm3.2} contains the proof of Theorem~\ref{thm:3.2} based on a localized cut-off technique and the complete cancellation of singular potential terms in the commutator identity.

%%%%%%%%%%%%%%%%%%%%%%%%%%%%%%%%%%%%%%%%%%%%%%%%%%%%%%%%%%%%%%%%%%%%%%%%%%%%%%%%%%%%%%%%%%%%%%%%%%%%%%%%%%%%%%%%%%%%%%%%%%%%%%%%%%%%%%%%%%%%%%%%%%%%%
%%%%%%%%%%%%%%%%%%%%%%%%%%%%%%%%%%%%%%%%%%%%%%%%%%%%%%%%%%%%%%%%%%%%%%%%%%%%%%%%%%%%%%%%%%%%%%%%%%%%%%%%%%%%%%%%%%%%%%%%%%%%%%%%%%%%%%%%%%%%%%%%%%%%%
\section{Preliminary Results}\label{sec:Prlm}

To rigorously define Schr\"odinger operators with distributional potentials $q \in H_{\mathrm{loc}}^{-1}(\R_+)$, we employ the theory of Shin--Zettl quasi-differential operators, originated in the pioneering works of D.~Shin \cite{Shin1938_I, Shin1938_II, Shin1938_III, Shin1940, Shin1943} and systematically developed by A.~Zettl \cite{Zttl1975} (see also the monographs \cite{EvMrksBook1999, Wei1987}). In recent years, this framework has found fruitful applications in the spectral analysis of differential operators with singular coefficients and measures \cite{EcGsNcTs2013, GszNclZnch2024, GrnMkh2010, GrnMkhPnk2013, HrMk2012, MkhGrnMlb2022, MkhMlb2013_mfat19n1, MkhMlb2013_mfat19n2, MkhMlb2018, MkhMlb2026}.

Under condition \eqref{eq:1.2}, one defines the associated Shin--Zettl matrix \cite{EvMrksBook1999, Shin1943, Zttl1975}:
\begin{equation*}\label{eq:2.1}
A(x) := \begin{pmatrix} Q(x) & 1 \\ -(Q^2(x) - s(x)) & -Q(x) \end{pmatrix} \in M_2(L_{1,\mathrm{loc}}(\R_+)).
\end{equation*}
The matrix $A(x)$ defines the quasi-derivatives of a function $u$ via:
\begin{align}
u^{[0]} &:= u,  \notag  \\ %\label{eq:2.2-0}
u^{[1]} &:= u' - Qu, \label{eq:2.2-1}\\
u^{[2]} &:= (u^{[1]})' + Qu^{[1]} + (Q^2 - s)u. \label{eq:2.2-2}
\end{align}
Then the formal differential expression $l[\cdot]$ is rigorously realized as the quasi-differential expression:
\begin{equation}\label{eq:2.2}
l[u] := -u^{[2]}, \qquad \Dom(l) := \left\{ u : \R_+ \to \C \;\middle|\; u, u^{[1]} \in \AC_{\mathrm{loc}}(\R_+) \right\}.
\end{equation}
Notice that when the potential $q$ is regular (e.g., continuous), expanding the brackets in \eqref{eq:2.2-2} shows that $l[u] = -(u' - Qu)' - Q(u' - Qu) - (Q^2 - s)u = -u'' + Q'u + Qu' - Qu' + Q^2u - Q^2u + su = -u'' + (s + Q')u$, which coincides with the classical expression $-u'' + qu$. In the general case, the equality holds in the sense of Schwartz distributions:
\[
\langle -u^{[2]}, \varphi \rangle = \langle -u'' + qu, \varphi \rangle \qquad \forall \varphi \in C_0^\infty(0, \infty).
\]

In the Hilbert space $L_2(\R_+)$, the quasi-differential expression \eqref{eq:2.2} gives rise to:

\noindent -- the \textit{maximal operator} $S$:
\begin{equation*}\label{eq:S-def}
  \begin{aligned}
    Su &:= l[u],\\
    \Dom(S) &:= \left\{ u \in L_2(\R_+) \;\middle|\; u, u^{[1]} \in \AC_{\mathrm{loc}}(\R_+), \; u(0) = 0, \; l[u] \in L_2(\R_+) \right\};
  \end{aligned}
\end{equation*}
-- the \textit{preminimal operator} $\dot{S}_0$, defined as the restriction of $S$ to compactly supported functions:
\begin{equation*}\label{eq:S0dot-def}
    \dot{S}_0 u := l[u], \quad \Dom(\dot{S}_0) := \left\{ u \in \Dom(S) \;\middle|\; \supp u \Subset \R_+ \right\},
\end{equation*}
where $\supp u \Subset \R_+$ means that there exists $b > 0$ such that $u(x) = 0$ for all $x \ge b$.

The formal Lagrange adjoint expression $l^+$ is determined by the adjoint matrix $A^+(x) := -J^{-1}A^*(x)J$, where $J = \begin{pmatrix} 0 & -1 \\ 1 & 0 \end{pmatrix}$ and $A^* = \overline{A}^T$. A direct computation yields the adjoint quasi-derivatives:
\begin{align*}
v^{\{0\}} &:= v, \\ %\label{eq:2.3-0}
v^{\{1\}} &:= v' - \overline{Q}v, \\ %\label{eq:2.3-1}
v^{\{2\}} &:= (v^{\{1\}})' + \overline{Q}v^{\{1\}} + (\overline{Q}^2 - \overline{s})v, % \label{eq:2.3-2}
\end{align*}
giving rise to the adjoint expression
\begin{equation}\label{eq:2.3}
l^+[v] := -v^{\{2\}}, \qquad \Dom(l^+) := \left\{ v : \R_+ \to \C \;\middle|\; v, v^{\{1\}} \in \AC_{\mathrm{loc}}(\R_+) \right\}.
\end{equation}
The corresponding maximal and preminimal operators generated by $l^+$ in $L_2(\R_+)$ are:
\begin{align*}
S^+ v &:= l^+[v], \notag\\
\Dom(S^+) &:= \left\{ v \in L_2(\R_+) \;\middle|\; v, v^{\{1\}} \in \AC_{\mathrm{loc}}(\R_+), \; v(0) = 0, \; l^+[v] \in L_2(\R_+) \right\}, \\ %\label{eq:Splus-def}
\dot{S}_0^+ v &:= l^+[v], \quad \Dom(\dot{S}_0^+) := \left\{ v \in \Dom(S^+) \;\middle|\; \supp v \Subset \R_+ \right\}. %\label{eq:S0plusdot-def}
\end{align*}

A fundamental tool in the analysis of these operators is the Lagrange identity \cite{EvMrksBook1999, Zttl1975}. For any $u \in \Dom(S)$ and $v \in \Dom(S^+)$ on any finite interval $[a, b] \subset \R_+$,
\begin{equation}\label{eq:Lagrange}
\int_a^b l[u]\bar{v}\, dx - \int_a^b u\,\overline{l^+[v]}\, dx = [u, v]_a^b := [u, v](b) - [u, v](a),
\end{equation}
where the sesquilinear boundary form is defined by
\begin{equation*}\label{eq:boundary-form}
[u, v](x) := u(x)\overline{v^{\{1\}}(x)} - u^{[1]}(x)\overline{v(x)}.
\end{equation*}
Moreover, for any $u \in \Dom(S)$ and $v \in \Dom(S^+)$, the limit as $b \to \infty$ exists and is finite:
\begin{equation*}\label{eq:limit-infty}
[u, v](\infty) := \lim_{x \to \infty} [u, v](x).
\end{equation*}
According to \cite{EvMrksBook1999, Zttl1975}, the domains of all four operators $S, \dot{S}_0, S^+, \dot{S}_0^+$ are dense in $L_2(\R_+)$.

The following theorem summarizes the fundamental structural properties of these operators.

\begin{theorem}[Structural Properties of Schr\"odinger Operators]\label{thm:2.1}
Let $q = s + Q'$ with $s \in L_{1,\mathrm{loc}}(\R_+)$ and $Q \in L_{2,\mathrm{loc}}(\R_+)$. Then the operators $S, \dot{S}_0, S^+, \dot{S}_0^+$ satisfy the following properties:
\begin{enumerate}
    \item[$1^\circ$.] Adjoint Relations:
    \begin{equation*}\label{eq:adjoints}
    (\dot{S}_0)^* = S^+, \qquad (\dot{S}_0^+)^* = S.
    \end{equation*}
    \item[$2^\circ$.] Closedness and Closability: The maximal operators $S$ and $S^+$ are closed. The preminimal operators $\dot{S}_0$ and $\dot{S}_0^+$ are closable, and their closures define the minimal operators:
    \begin{equation*}\label{eq:min-def}
    S_0 := \overline{\dot{S}_0}, \qquad S_0^+ := \overline{\dot{S}_0^+},
    \end{equation*}
    satisfying
    \begin{equation*}\label{eq:min-adjoints}
    (S_0)^* = S^+, \qquad (S_0^+)^* = S.
    \end{equation*}
    \item[$3^\circ$.] Characterization of Minimal Domains:
    \begin{align}
    \Dom(S_0) &= \left\{ u \in \Dom(S) \;\middle|\; [u, v](\infty) = 0 \quad \forall v \in \Dom(S^+) \right\}, \label{eq:DomS0}\\
    \Dom(S_0^+) &= \left\{ v \in \Dom(S^+) \;\middle|\; [u, v](\infty) = 0 \quad \forall u \in \Dom(S) \right\}. \label{eq:DomS0plus}
    \end{align}
    \item[$4^\circ$.] Sobolev Embeddings:
    \begin{equation*}\label{eq:Sobolev}
    \Dom(S), \Dom(S^+) \subset W_{2,\mathrm{loc}}^1(\R_+), \qquad \Dom(\dot{S}_0), \Dom(\dot{S}_0^+) \subset W_{2,\mathrm{comp}}^1(\R_+).
    \end{equation*}
    \item[$5^\circ$.] Action of Complex Conjugation: Let $C : L_2(\R_+) \to L_2(\R_+)$ denote the anti-linear involution of complex conjugation: $(Cu)(x) := \overline{u(x)}$. Then $C$ is an anti-linear unitary involution ($C = C^* = C^{-1}$) mapping:
    \begin{equation*}\label{eq:C-action}
    C \Dom(\dot{S}_0) = \Dom(\dot{S}_0^+), \quad C \Dom(S_0) = \Dom(S_0^+), \quad C \Dom(S) = \Dom(S^+).
    \end{equation*}
    In particular, if $Q$ and $s$ are real-valued, then $\Dom(\dot{S}_0) = \Dom(\dot{S}_0^+)$, $\Dom(S_0) = \Dom(S_0^+)$, and $\Dom(S) = \Dom(S^+)$.
    \item[$6^\circ$.] $C$-Symmetry: The operators $S_0$ and $S_0^+$ are $C$-symmetric:
    \begin{equation*}\label{eq:C-symm}
    C S_0 C  \subset (S_0)^*, \qquad C S_0^+ C  \subset  (S_0^+)^*.
    \end{equation*}
      \item[$7^\circ$.] Preminimal Quadratic Form: The quadratic form $\dot{t}[u] := (\dot{S}_0 u, u)_{L_2}$ is densely defined on $\Dom(\dot{t}) := \Dom(\dot{S}_0)$ and admits the exact representation:
    \begin{equation}\label{eq:form-repr}
    \dot{t}[u] = \int_{\R_+} |u'|^2 dx - \int_{\R_+} Q(u'\bar{u} + u\bar{u}') dx + \int_{\R_+} s|u|^2 dx.
    \end{equation}
\end{enumerate}
\end{theorem}

\begin{corollary}\label{cor:2.2}
The preminimal operators $\dot{S}_0$ and $\dot{S}_0^+$ are accretive simultaneously. Furthermore, for every $u \in \Dom(\dot{S}_0)$ and $v := Cu \in \Dom(\dot{S}_0^+)$, their quadratic forms share identical real parts:
\begin{equation*}\label{eq:real-parts}
\mathrm{Re}\,(\dot{S}_0 u, u)_{L_2(\R_+)} = \mathrm{Re}\,(\dot{S}_0^+ v, v)_{L_2(\R_+)}.
\end{equation*}
\end{corollary}

Recall that an operator $T$ in a Hilbert space $H$ is called \emph{quasi-accretive} if there exists $\alpha \in \R$ such that $T + \alpha I$ is accretive (and \emph{quasi-$m$-accretive} if $T + \alpha I$ is $m$-accretive). An operator $T$ is called \emph{sectorial} with vertex $\gamma \in \R$ and semi-angle $\theta \in [0, \pi/2)$ if its numerical range satisfies $\Theta(T - \gamma I) \subset \Lambda_\theta$, which is equivalent to the inequality:
\begin{equation*}\label{eq:sectorial-def}
\mathrm{Re}\,(Tu, u)_H - k |\mathrm{Im}\,(Tu, u)_H| \ge \gamma \|u\|_H^2 \qquad \forall u \in \Dom(T),
\end{equation*}
where $k := \cot \theta > 0$. An operator is called \emph{$m$-sectorial} if it is sectorial and quasi-$m$-accretive \cite{Kato1995, Schm2012}.

%%%%%%%%%%%%%%%%%%%%%%%%%%%%%%%%%%%%%%%%%%%%%%%%%%%%%%%%%%%%%%%%%%%%%%%%%%%%%%%%%%%%%%%%%%%%%%%%%%%%%%%%%%%%%%%%%%%%%%%%%%%%%%%%%%%%%%%%%%%%%%%%%%%%%
%%%%%%%%%%%%%%%%%%%%%%%%%%%%%%%%%%%%%%%%%%%%%%%%%%%%%%%%%%%%%%%%%%%%%%%%%%%%%%%%%%%%%%%%%%%%%%%%%%%%%%%%%%%%%%%%%%%%%%%%%%%%%%%%%%%%%%%%%%%%%%%%%%%%%
\subsection{Regular Quasi-Differential Operators on Finite Intervals}

Let $\Delta = [\alpha, \beta] \subset \R_+$ be a compact interval. In $L_2(\Delta)$, the quasi-differential expressions $l_\Delta$ and $l_\Delta^+$ are defined by the formulas \eqref{eq:2.2} and \eqref{eq:2.3} restricted to $\Delta$. They generate the maximal and minimal regular operators:
\begin{align*}
S_\Delta u &:= l_\Delta[u], \quad \Dom(S_\Delta) := \left\{ u \in L_2(\Delta) \;\middle|\; u, u^{[1]} \in \AC(\Delta), \; l_\Delta[u] \in L_2(\Delta) \right\}, \\ %\label{eq:S-Delta}
S_{0,\Delta} u &:= l_\Delta[u], \quad \Dom(S_{0,\Delta}) := \left\{ u \in \Dom(S_\Delta) \;\middle|\; u^{[j]}(\alpha) = u^{[j]}(\beta) = 0, \; j = 0, 1 \right\}, % \label{eq:S0-Delta}
\end{align*}
and correspondingly for the adjoint:
\begin{align*}
S_\Delta^+ v &:= l_\Delta^+[v], \quad \Dom(S_\Delta^+) := \left\{ v \in L_2(\Delta) \;\middle|\; v, v^{\{1\}} \in \AC(\Delta), \; l_\Delta^+[v] \in L_2(\Delta) \right\}, \\ %\label{eq:Splus-Delta}
S_{0,\Delta}^+ v &:= l_\Delta^+[v], \quad \Dom(S_{0,\Delta}^+) := \left\{ v \in \Dom(S_\Delta^+) \;\middle|\; v^{\{j\}}(\alpha) = v^{\{j\}}(\beta) = 0, \; j = 0, 1 \right\}. %\label{eq:S0plus-Delta}
\end{align*}

\begin{propositionEng}[Zettl \cite{Zttl1975}, Theorems 10, 11]\label{prop:regular}
The minimal operators $S_{0,\Delta}$ and $S_{0,\Delta}^+$ are densely defined, closed operators in $L_2(\Delta)$, and the exact adjoint relations hold:
\begin{equation*}\label{eq:regular-adj}
(S_{0,\Delta})^* = S_\Delta^+, \qquad (S_{0,\Delta}^+)^* = S_\Delta, \qquad (S_\Delta)^* = S_{0,\Delta}^+, \qquad (S_\Delta^+)^* = S_{0,\Delta}.
\end{equation*}
\end{propositionEng}

\subsection{Proof of Theorem~\ref{thm:2.1}}

We establish each statement of Theorem~\ref{thm:2.1} systematically, compare with \cite{EvMrksBook1999, NmrkRus1969, Zttl1975}.

\smallskip
\noindent \emph{Proof of Statement $1^\circ$} (Adjoint Relations). 
We prove $(\dot{S}_0)^* = S^+$.

\emph{Step 1: The inclusion $S^+ \subset (\dot{S}_0)^*$.} Let $u \in \Dom(\dot{S}_0)$ and $v \in \Dom(S^+)$. Since $\supp u \Subset \R_+$, there exists $b > 0$ such that $u(x) = 0$ for all $x \ge b$. Since $u, u^{[1]} \in \AC_{\mathrm{loc}}(\R_+)$, we have $u(b) = u^{[1]}(b) = 0$. Applying the Lagrange identity \eqref{eq:Lagrange} on $[0, b]$ yields:
\[
(\dot{S}_0 u, v)_{L_2} - (u, S^+ v)_{L_2} = \int_0^b \left( l[u]\bar{v} - u\overline{l^+[v]} \right) dx = [u, v](b) - [u, v](0).
\]
At $x = b$, $[u, v](b) = u(b)\overline{v^{\{1\}}(b)} - u^{[1]}(b)\overline{v(b)} = 0$. At $x = 0$, both $u$ and $v$ satisfy the Dirichlet condition $u(0) = v(0) = 0$, so:
\[
[u, v](0) = 0 \cdot \overline{v^{\{1\}}(0)} - u^{[1]}(0) \cdot 0 = 0.
\]
Consequently, for all $u \in \Dom(\dot{S}_0)$ we have $(\dot{S}_0 u, v)_{L_2} = (u, S^+ v)_{L_2}$, which proves that $v \in \Dom((\dot{S}_0)^*)$ with $(\dot{S}_0)^* v = S^+ v$, establishing $S^+ \subset (\dot{S}_0)^*$.

\emph{Step 2: Local regularity of elements in $\Dom((\dot{S}_0)^*)$.} Let $v \in \Dom((\dot{S}_0)^*)$, meaning there exists $f \in L_2(\R_+)$ such that
\begin{equation}\label{eq:adj-def}
(\dot{S}_0 u, v)_{L_2(\R_+)} = (u, f)_{L_2(\R_+)} \qquad \forall u \in \Dom(\dot{S}_0).
\end{equation}
Consider an arbitrary compact subinterval $\Delta = [\alpha, \beta] \subset (0, \infty)$ with $0 < \alpha < \beta < \infty$. For any function $u \in \Dom(S_{0,\Delta})$, extending $u$ by zero outside $\Delta$ yields a function in $\Dom(\dot{S}_0)$. Identity \eqref{eq:adj-def} then gives:
\[
(S_{0,\Delta} u, v|_\Delta)_{L_2(\Delta)} = (u, f|_\Delta)_{L_2(\Delta)} \qquad \forall u \in \Dom(S_{0,\Delta}).
\]
By Proposition~\ref{prop:regular}, $(S_{0,\Delta})^* = S_\Delta^+$. Hence, $v|_\Delta \in \Dom(S_\Delta^+)$, which implies $v|_\Delta \in \AC(\Delta)$, $v^{\{1\}}|_\Delta \in \AC(\Delta)$, and $l^+[v]|_\Delta = f|_\Delta$. Since $\Delta \subset (0, \infty)$ is arbitrary, we obtain $v, v^{\{1\}} \in \AC_{\mathrm{loc}}(0, \infty)$ and $l^+[v] = f \in L_2(\R_+)$.

\emph{Step 3: Regularity at $x = 0$ and the boundary condition $v(0) = 0$.} Now consider the interval $\Delta_0 = [0, b]$ for some $b > 0$. The vector $(v, v^{\{1\}})^T$ satisfies almost everywhere on $[0, b]$ the first-order non-homogeneous system:
\begin{equation*}\label{eq:caratheodory}
\begin{pmatrix} v \\ v^{\{1\}} \end{pmatrix}' = \begin{pmatrix} \overline{Q} & 1 \\ -(\overline{Q}^2 - \overline{s}) & -\overline{Q} \end{pmatrix} \begin{pmatrix} v \\ v^{\{1\}} \end{pmatrix} + \begin{pmatrix} 0 \\ -f \end{pmatrix}.
\end{equation*}
Since the coefficient matrix has entries in $L_1[0, b]$ and $f \in L_2[0, b]$, the classical Carath\'eodory theorem \cite{Zttl1975} guarantees that $v$ and $v^{\{1\}}$ extend continuously to $x = 0$, so that $v, v^{\{1\}} \in \AC[0, b]$. Therefore, $v, v^{\{1\}} \in \AC_{\mathrm{loc}}(\R_+)$.

Applying the Lagrange identity to any $u \in \Dom(\dot{S}_0)$ with $\supp u \subset [0, b]$ and our function $v$:
\[
\int_0^b \left( l[u]\bar{v} - u\overline{l^+[v]} \right) dx = [u, v](b) - [u, v](0) = -[u, v](0).
\]
On the other hand, by \eqref{eq:adj-def}, the left-hand side equals $(\dot{S}_0 u, v) - (u, f) = 0$. Hence, for every $u \in \Dom(\dot{S}_0)$,
\begin{equation}\label{eq:boundary-zero}
[u, v](0) = u(0)\overline{v^{\{1\}}(0)} - u^{[1]}(0)\overline{v(0)} = 0.
\end{equation}
By \cite{Zttl1975}, for any prescribed complex numbers $c_0, c_1 \in \C$, there exists a function $u \in \Dom(l)$ with compact support in $[0, b]$ such that $u(0) = c_0$ and $u^{[1]}(0) = c_1$. To satisfy $u \in \Dom(\dot{S}_0)$, we set $c_0 = 0$, while $c_1 = u^{[1]}(0)$ can be chosen arbitrarily. Choosing a test function $u_0 \in \Dom(\dot{S}_0)$ with $u_0(0) = 0$ and $u_0^{[1]}(0) = 1$, substitution into \eqref{eq:boundary-zero} yields:
\[
[u_0, v](0) = 0 \cdot \overline{v^{\{1\}}(0)} - 1 \cdot \overline{v(0)} = -\overline{v(0)} = 0 \implies v(0) = 0.
\]
Thus, $v$ satisfies the Dirichlet boundary condition $v(0) = 0$, which proves that $v \in \Dom(S^+)$ and $(\dot{S}_0)^* v = l^+[v] = S^+ v$. Hence, $(\dot{S}_0)^* \subset S^+$, which together with Step~1 establishes $(\dot{S}_0)^* = S^+$. The equality $(\dot{S}_0^+)^* = S$ is proved identically by exchanging $l$ and $l^+$.

\smallskip
\noindent \emph{Proof of Statement $2^\circ$} (Closedness and Closures).
As adjoints of densely defined operators, $S = (\dot{S}_0^+)^*$ and $S^+ = (\dot{S}_0)^*$ are automatically closed operators. Since $S$ and $S^+$ are densely defined, their adjoints $\dot{S}_0$ and $\dot{S}_0^+$ are closable. Their closures are given by $S_0 = \overline{\dot{S}_0} = (\dot{S}_0)^{**} = (S^+)^*$, and $S_0^+ = \overline{\dot{S}_0^+} = (\dot{S}_0^+)^{**} = S^*$. Taking adjoints immediately gives $(S_0)^* = S^+$ and $(S_0^+)^* = S$.

\smallskip
\noindent \emph{Proof of Statement $3^\circ$} (Minimal Domains).
By definition, $\Dom(S_0) = \Dom((S^+)^*)$:
\[
\Dom(S_0) = \left\{ u \in \Dom(S) \;\middle|\; (Su, v)_{L_2} = (u, S^+ v)_{L_2} \quad \forall v \in \Dom(S^+) \right\}.
\]
For any $u \in \Dom(S)$ and $v \in \Dom(S^+)$, the Lagrange identity on $[0, b]$ gives $(Su, v)_{L_2[0, b]} - (u, S^+ v)_{L_2[0, b]} = [u, v](b) - [u, v](0) = [u, v](b)$, since $u(0) = v(0) = 0$. Letting $b \to \infty$, we obtain Green's formula on the half-line:
\[
(Su, v)_{L_2(\R_+)} - (u, S^+ v)_{L_2(\R_+)} = [u, v](\infty).
\]
Hence, $(Su, v) = (u, S^+ v)$ for all $v \in \Dom(S^+)$ if and only if $[u, v](\infty) = 0$, which proves \eqref{eq:DomS0}. The formula \eqref{eq:DomS0plus} for $\Dom(S_0^+)$ follows symmetrically.

\smallskip
\noindent \emph{Proof of Statement $4^\circ$} (Sobolev Embeddings).
Let $u \in \Dom(S)$. By definition of quasi-derivatives, $u^{[1]} = u' - Qu \in \AC_{\mathrm{loc}}(\R_+) \subset L_{2,\mathrm{loc}}(\R_+)$. Since $u \in \AC_{\mathrm{loc}}(\R_+)$, $u$ is continuous and hence locally bounded on any compact interval $[0, a]$. Since $Q \in L_{2,\mathrm{loc}}(\R_+)$, the product $Qu$ belongs to $L_{2,\mathrm{loc}}(\R_+)$. Therefore,
\[
u' = u^{[1]} + Qu \in L_{2,\mathrm{loc}}(\R_+).
\]
Since $u \in L_2(\R_+)$ and $u' \in L_{2,\mathrm{loc}}(\R_+)$, we have $u \in W_{2,\mathrm{loc}}^1(\R_+)$. The same reasoning applies to $v \in \Dom(S^+)$, where $v' = v^{\{1\}} + \overline{Q}v \in L_{2,\mathrm{loc}}(\R_+)$. For compactly supported functions in $\Dom(\dot{S}_0)$ and $\Dom(\dot{S}_0^+)$, the derivative $u'$ has compact support, yielding $u \in W_{2,\mathrm{comp}}^1(\R_+)$.

\smallskip
\noindent \emph{Proof of Statement $5^\circ$} ($C$-Action).
Let $u \in \Dom(S)$ and $w := Cu = \bar{u}$. Clearly $w \in L_2(\R_+)$ and $w(0) = \overline{u(0)} = 0$. For the first adjoint quasi-derivative:
\[
w^{\{1\}} = w' - \overline{Q}w = \overline{u}' - \overline{Q}\bar{u} = \overline{u' - Qu} = \overline{u^{[1]}}.
\]
Since complex conjugation preserves absolute continuity, $u^{[1]} \in \AC_{\mathrm{loc}}(\R_+)$ implies $w^{\{1\}} \in \AC_{\mathrm{loc}}(\R_+)$. Next, compute the second adjoint quasi-derivative:
\begin{align*}
l^+[w] &= - \left[ (w^{\{1\}})' + \overline{Q}w^{\{1\}} + (\overline{Q}^2 - \overline{s})w \right] \\
&= - \left[ \overline{(u^{[1]})'} + \overline{Q}\,\overline{u^{[1]}} + (\overline{Q}^2 - \overline{s})\bar{u} \right] = \overline{- \left[ (u^{[1]})' + Qu^{[1]} + (Q^2 - s)u \right]} = \overline{l[u]}.
\end{align*}
Since $l[u] \in L_2(\R_+)$, we have $l^+[w] \in L_2(\R_+)$, which implies $w = Cu \in \Dom(S^+)$. Since $C^2 = I$, the reverse inclusion $C \Dom(S^+) \subset \Dom(S)$ holds identically, proving $C \Dom(S) = \Dom(S^+)$. 

If $u$ has compact support, $\bar{u}$ has the same support, so $C \Dom(\dot{S}_0) = \Dom(\dot{S}_0^+)$. For the minimal operators, $C \Dom(S_0) = \Dom(S_0^+)$ follows from the domain characterizations \eqref{eq:DomS0}--\eqref{eq:DomS0plus} and the identity $[Cu, Cv](x) = -\overline{[u, v](x)}$.

\smallskip
\noindent \emph{Proof of Statement $6^\circ$} ($C$-Symmetry).
By Statement $5^\circ$, for every $u \in \Dom(S_0)$, we have $(C S_0 C)u = C S_0 \bar{u} = \overline{l[\bar{u}]} = l^+[u] = S_0^+ u$. Since $S_0^+ \subset S^+$, we obtain $C S_0 C = S_0^+ \subset (S_0)^*$. Similarly, $C S_0^+ C = S_0 \subset S= (S_0^+)^*$.

\smallskip
\noindent \emph{Proof of Statement $7^\circ$} (Quadratic Form).
Let $u \in \Dom(\dot{S}_0)$ with $\supp u \subset [0, b]$ and $u(0) = 0$. By definition,
\begin{align*}
\dot{t}[u] &:= (\dot{S}_0 u, u)_{L_2} = \int_0^b l[u]\bar{u}\, dx \\
&= \int_0^b \left[ -(u^{[1]})'\bar{u} - Qu^{[1]}\bar{u} - (Q^2 - s)|u|^2 \right] dx.
\end{align*}
Integrating the first term by parts and using $u(0) = u(b) = 0$:
\begin{align*}
-\int_0^b (u^{[1]})'\bar{u}\, dx &= -u^{[1]}(x)\overline{u(x)}\Big|_0^b + \int_0^b u^{[1]}\bar{u}'\, dx \\
&= \int_0^b (u' - Qu)\bar{u}'\, dx = \int_0^b |u'|^2 dx - \int_0^b Qu\bar{u}'\, dx.
\end{align*}
In the second term, substitute $u^{[1]} = u' - Qu$:
\[
-\int_0^b Qu^{[1]}\bar{u}\, dx = -\int_0^b Q(u' - Qu)\bar{u}\, dx = -\int_0^b Qu'\bar{u}\, dx + \int_0^b Q^2|u|^2 dx.
\]
The third term is $-\int_0^b (Q^2 - s)|u|^2 dx = -\int_0^b Q^2|u|^2 dx + \int_0^b s|u|^2 dx$.

Summing all three contributions, the terms $+Q^2|u|^2$ and $-Q^2|u|^2$ cancel each other out identically:
\[
\dot{t}[u] = \int_0^b |u'|^2 dx - \int_0^b Q(u'\bar{u} + u\bar{u}') dx + \int_0^b s|u|^2 dx.
\]
Extending the integration to $\R_+$ via the support of $u$, we arrive at formula \eqref{eq:form-repr}. 

The proof is complete.

%%%%%%%%%%%%%%%%%%%%%%%%%%%%%%%%%%%%%%%%%%%%%%%%%%%%%%%%%%%%%%%%%%%%%%%%%%%%%%%%%%%%%%%%%%%%%%%%%%%%%%%%%%%%%%%%%%%%%%%%%%%%%%%%%%%%%%%%%%%%%%%%%%%%%
%%%%%%%%%%%%%%%%%%%%%%%%%%%%%%%%%%%%%%%%%%%%%%%%%%%%%%%%%%%%%%%%%%%%%%%%%%%%%%%%%%%%%%%%%%%%%%%%%%%%%%%%%%%%%%%%%%%%%%%%%%%%%%%%%%%%%%%%%%%%%%%%%%%%%
\section{Main Theorems}\label{sec:MnRslts}

Decompose the coefficients into their real and imaginary parts:
\begin{align*}
Q(x) &= Q_0(x) + i Q_1(x), \qquad Q_0 := \mathrm{Re}\,Q, \quad Q_1 := \mathrm{Im}\,Q, \\ %\label{eq:Q-decomp}
s(x) &= s_0(x) + i s_1(x), \qquad s_0 := \mathrm{Re}\,s, \quad s_1 := \mathrm{Im}\,s. %\label{eq:s-decomp}
\end{align*}

\begin{theorem}[Accretivity and Sectoriality Conditions]\label{thm:3.1}
Under the general assumptions \eqref{eq:1.2} on the potential $q$, the following assertions hold:
\begin{enumerate}
    \item[\textbf{(I)}] Accretivity: The preminimal operators $\dot{S}_0$ and $\dot{S}_0^+$ are accretive in $L_2(\R_+)$ if
    \begin{equation}\label{eq:accretive-cond}
    s_0(x) \ge |Q_0(x)|^2 \qquad \text{for a.e. } x \in \R_+.
    \end{equation}
    \item[\textbf{(II)}] Sectoriality: If for some $k > 0$ there exists a non-positive constant $\gamma = \gamma(k) \le 0$ such that
    \begin{equation}\label{eq:sectorial-cond}
    s_0(x) - \gamma \ge \left(|Q_0(x)| + k|Q_1(x)|\right)^2 + k|s_1(x)| \qquad \text{for a.e. } x \in \R_+,
    \end{equation}
    then $\dot{S}_0$ and $\dot{S}_0^+$ are sectorial with vertex $\gamma$ and semi-angle $\theta = \arctan(1/k) \in (0, \pi/2)$.
    \item[\textbf{(III)}] Quasi-Accretivity for Measure Potentials: If the real part of the potential $q$ is a Radon measure satisfying the generalized Brinck condition \cite{Brnk1959, MkhMlb2018}:
    \begin{equation}\label{eq:brinck-real}
    \int_J s_0(x)\, dx + \int_J dQ_0(x) \ge -C \qquad (C > 0)
    \end{equation}
    for every interval $J \subset \R_+$ with length $|J| \le 1$, then $\dot{S}_0$ and $\dot{S}_0^+$ are quasi-accretive with shift $\alpha = 2C^2$.
    \item[\textbf{(IV)}] Sectoriality for Complex Measure Potentials: If, in addition, the imaginary part of $q$ is a Radon measure satisfying
    \begin{equation}\label{eq:brinck-imag}
    \left| \int_J s_1(x)\, dx + \int_J dQ_1(x) \right| \le C \qquad \forall J \subset \R_+, \; |J| \le 1,
    \end{equation}
    then the operators $\dot{S}_0$ and $\dot{S}_0^+$ are sectorial with vertex $\gamma = -4C^2(1+k)^2$ and semi-angle $\theta = \arctan(1/k)$ for any $k > 0$.
\end{enumerate}
\end{theorem}

The next theorem is the central result of this paper, establishing that for Schr\"odinger operators on the half-line with Dirichlet boundary condition at $x = 0$, accretivity is equivalent to $m$-accretivity. As demonstrated in \cite{MkhMlb2026}, this remarkable property fails for operators with first-order convection/drift terms on the whole line~$\R$.

\begin{theorem}[Global Uniqueness and $m$-Accretivity]\label{thm:3.2}
If the preminimal operator $\dot{S}_0$ is accretive (respectively, sectorial), then the minimal operators $S_0$ and $S_0^+$ are $m$-accretive (respectively, $m$-sectorial). 

Moreover, the minimal and maximal operators coincide:
\begin{equation*}\label{eq:coincidence}
S_0 = S \qquad \text{and} \qquad S_0^+ = S^+,
\end{equation*}
which implies that the boundary form vanishes identically at infinity:
\begin{equation}\label{eq:vanish-boundary}
[u, v](\infty) = 0 \quad \forall u \in \Dom(S), \; \forall v \in \Dom(S^+).
\end{equation}
Consequently, $-S$ generates in $L_2(\R_+)$ a strongly continuous contraction semi\-group$\{e^{-tS}\}_{t \ge 0}$ (holomorphic if $S$ is sectorial).
\end{theorem}

%%%%%%%%%%%%%%%%%%%%%%%%%%%%%%%%%%%%%%%%%%%%%%%%%%%%%%%%%%%%%%%%%%%%%%%%%%%%%%%%%%%%%%%%%%%%%%%%%%%%%%%%%%%%%%%%%%%%%%%%%%%%%%%%%%%%%%%%%%%%%%%%
%%%%%%%%%%%%%%%%%%%%%%%%%%%%%%%%%%%%%%%%%%%%%%%%%%%%%%%%%%%%%%%%%%%%%%%%%%%%%%%%%%%%%%%%%%%%%%%%%%%%%%%%%%%%%%%%%%%%%%%%%%%%%%%%%%%%%%%%%%%%%%%%
%Нагадаємо (див., наприклад, \cite{ArlSchm2026, Glzmn1966}), що щільно визначений замкнений оператор $\mathrm{S}$ в %гільбертовому просторі $\mathrm{H}$ називають $\mathrm{C}$-самоспряженим стосовно антилінійної інволюції $\mathrm{C}$, якщо %$\mathrm{C}\mathrm{T}\mathrm{C}=\mathrm{T}^{\ast}$. Як встановлено в теоремі~\ref{th_Prprts} (пункт $6^{0}$) мінімальні %оператори $\mathrm{S}_{0}$ і $\mathrm{S}_{0}^{+}$ задовольняють включення
%\begin{equation*}
%   \mathrm{C}\mathrm{S}_{0}\mathrm{C}\subset \mathrm{S}_{0}^{\ast}, \qquad 
%   \mathrm{C}\mathrm{S}_{0}^{+}\mathrm{C}\subset (\mathrm{S}_{0}^{+})^{\ast},
%\end{equation*}
%тобто є $\mathrm{C}$-симетричними, щодо інволюції комплексного спряження в функціональному просторі %$L^{2}(\mathbb{R}_{+})$. Легко бачити, що ці включення є рівностями, тоді і лише тоді, коли $\mathrm{S}_{0}=\mathrm{S}$, %$\mathrm{S}_{0}^{+}=\mathrm{S}^{+}$. Тому з теорем~\ref{th_mAccr} і \ref{th_Accr} випливає

Let us remind (see, for an example, \cite{ArlSchm2026, Glzmn1966}) that a densely defined closed operator $T$ in a Hilbert space $H$ is called $C$-self-adjoint with respect to the anti-linear involution $C$ if $C T C=T^*$. As established in Theorem~\ref{thm:2.1}~($6^0$) the minimal operators $S_{0}$ and $S_{0}^{+}$ are satisfied the inclusions
\begin{equation*}
   C S_{0}C \subset S_{0}^{\ast}, \qquad 
   C S_{0}^{+}C \subset (S_{0}^{+})^{*},
\end{equation*}
i. e., they are $C$-symmetric with respect to the complex adjoint involution $C$ in the functional space $L_{2}(\mathbb{R}_{+})$. It is easy to see that the inclusions are equalities indeed if and only if $S_0=S$ and $S_0^+=S^+$. Therefore, from Theorems~\ref{thm:3.1} and \ref{thm:3.2} it follows 

\begin{theorem}[$C$-Self-Adjointness]\label{thm:3.3}
If the preminimal operator $\dot{S}_0$ is quasi-accretive, then the minimal operators $S_0$ and $S_0^+$ are $C$-self-adjoint with respect to complex conjugation:
\begin{equation*}\label{eq:C-selfadj}
C S_0 C = S_0^* , \qquad C S_0^+ C = (S_0^+)^*.
\end{equation*}
In particular, this holds if there exists a constant $\alpha > 0$ such that
\begin{equation}\label{eq:C-selfadj-cond}
s_0(x) + \alpha \ge |Q_0(x)|^2 \qquad \text{for a.e. } x \in \R_+.
\end{equation}
Moreover, the residual spectrum of $S_0$ is empty:
\[
\sigma_r(S_0) = \emptyset.
\]
\end{theorem}

%%%%%%%%%%%%%%%%%%%%%%%%%%%%%%%%%%%%%%%%%%%%%%%%%%%%%%%%%%%%%%%%%%%%%%%%%%%%%%%%%%%%%%%%%%%%%%%%%%%%%%%%%%%%%%%%%%%%%%%%%%%%%%%%%%%%%%%%%%%%%%%%%%%%%
%%%%%%%%%%%%%%%%%%%%%%%%%%%%%%%%%%%%%%%%%%%%%%%%%%%%%%%%%%%%%%%%%%%%%%%%%%%%%%%%%%%%%%%%%%%%%%%%%%%%%%%%%%%%%%%%%%%%%%%%%%%%%%%%%%%%%%%%%%%%%%%%%%%%%
\section{Proof of Theorem~\ref{thm:3.1}}\label{sec:thm3.1}

\noindent \emph{Part (I):} Accretivity.
Let $u \in \Dom(\dot{S}_0)$. Decompose the quadratic form \eqref{eq:form-repr} into its real and imaginary parts:
\[
\dot{t}[u] = \dot{h}[u] + i \dot{f}[u], \qquad \dot{h}[u] := \mathrm{Re}\,\dot{t}[u], \quad \dot{f}[u] := \mathrm{Im}\,\dot{t}[u].
\]
Notice that $u'\bar{u} + u\bar{u}' = 2\,\mathrm{Re}(u'\bar{u}) = (|u|^2)'$ is purely real. Therefore,
\begin{align*}
Q(u'\bar{u} + u\bar{u}') &= (Q_0 + iQ_1)(|u|^2)' = Q_0(|u|^2)' + i Q_1(|u|^2)', \\
s|u|^2 &= (s_0 + is_1)|u|^2 = s_0|u|^2 + i s_1|u|^2.
\end{align*}
This yields the exact expressions:
\begin{align}
\dot{h}[u] &= \int_{\R_+} |u'|^2 dx - 2\int_{\R_+} Q_0\,\mathrm{Re}(u'\bar{u})\, dx + \int_{\R_+} s_0|u|^2 dx, \label{eq:h-expr}\\
\dot{f}[u] &= -2\int_{\R_+} Q_1\,\mathrm{Re}(u'\bar{u})\, dx + \int_{\R_+} s_1|u|^2 dx. \label{eq:f-expr}
\end{align}
Applying the Cauchy--Schwarz inequality $|\mathrm{Re}(u'\bar{u})| \le |u'||u|$ to \eqref{eq:h-expr}:
\begin{align*}
\dot{h}[u] &\ge \int_{\R_+} \left( |u'|^2 - 2|Q_0||u||u'| + s_0|u|^2 \right) dx \\
&= \int_{\R_+} \left( |u'| - |Q_0||u| \right)^2 dx + \int_{\R_+} \left( s_0(x) - |Q_0(x)|^2 \right) |u|^2 dx.
\end{align*}
Under condition \eqref{eq:accretive-cond}, $s_0(x) \ge |Q_0(x)|^2$ a.e. on $\R_+$, so both integrand terms are non-negative. Hence, $\dot{h}[u] \ge 0$, which proves the accretivity of $\dot{S}_0$. By Corollary~\ref{cor:2.2}, $\dot{S}_0^+$ is also accretive.

\smallskip
\noindent \emph{Part (II):} Sectoriality.
From \eqref{eq:f-expr}, we estimate the imaginary part:
\[
|\dot{f}[u]| \le 2\int_{\R_+} |Q_1||u||u'|\, dx + \int_{\R_+} |s_1||u|^2 dx.
\]
Then, for any $k > 0$ and $\gamma \le 0$:
\begin{align*}
\dot{h}[u] &- \gamma \|u\|_{L_2}^2 - k |\dot{f}[u]| \\
&\ge \int_{\R_+} \left[ |u'|^2 - 2\left(|Q_0| + k|Q_1|\right)|u||u'| + \left( s_0 - k|s_1| - \gamma \right)|u|^2 \right] dx \\
&= \int_{\R_+} \left( |u'| - (|Q_0| + k|Q_1|)|u| \right)^2 dx + \int_{\R_+} \left[ s_0 - \gamma - \left(|Q_0| + k|Q_1|\right)^2 - k|s_1| \right] |u|^2 dx.
\end{align*}
By assumption \eqref{eq:sectorial-cond}, the term in the second brackets is non-negative almost everywhere. Thus,
\[
\dot{h}[u] - \gamma \|u\|_{L_2}^2 - k |\dot{f}[u]| \ge 0 \qquad \forall u \in \Dom(\dot{S}_0),
\]
which proves sectoriality with vertex $\gamma$ and semi-angle $\theta = \arctan(1/k)$. By Corollary~\ref{cor:2.2}, $\dot{S}_0^+$ satisfies the same sectorial bound.

\smallskip
\noindent \emph{Part (III):} Quasi-Accretivity for Measure Potentials.
Introduce the primitives:
\begin{equation*}\label{eq:primitives}
\sigma_0(x) := \int_0^x s_0(t)\, dt, \quad \sigma_1(x) := \int_0^x s_1(t)\, dt, \qquad \widetilde{Q}_0 := Q_0 + \sigma_0, \quad \widetilde{Q}_1 := Q_1 + \sigma_1.
\end{equation*}
If $\mathrm{Re}\,q$ is a Radon measure, then $\widetilde{Q}_0 \in BV_{\mathrm{loc}}(\R_+)$, $\mathrm{Re}\,q = \widetilde{Q}_0'$, and the Brinck condition \eqref{eq:brinck-real} becomes:
\begin{equation}\label{eq:brinck-compact}
\int_J d\widetilde{Q}_0(x) \ge -C \qquad (C > 0) \quad \forall J \subset \R_+, \; |J| \le 1.
\end{equation}
For $u \in \Dom(\dot{S}_0)$, integration by parts gives:
\[
\dot{h}[u] = \int_{\R_+} |u'|^2 dx + \int_{\R_+} |u|^2 d\widetilde{Q}_0.
\]
By \cite[Theorem 2.4]{MkhMlb2018}, under condition \eqref{eq:brinck-compact}, the form bound holds:
\[
\int_{\R_+} |u|^2 d\widetilde{Q}_0 \ge -C\epsilon \int_{\R_+} |u'|^2 dx - 2C\epsilon^{-1} \int_{\R_+} |u|^2 dx \qquad \forall \epsilon \in (0, 1].
\]
Setting $\epsilon = 1/C \le 1$ (for $C \ge 1$), we obtain $1 - C\epsilon = 0$, so:
\[
\dot{h}[u] \ge -2C^2 \|u\|_{L_2}^2,
\]
which proves that $\dot{S}_0$ is quasi-accretive with $\alpha = 2C^2$.

\smallskip
\noindent \emph{Part (IV):} Sectoriality for Measure Potentials.
If $\mathrm{Im}\,q$ is also a Radon measure, $\widetilde{Q}_1 \in BV_{\mathrm{loc}}(\R_+)$, and the condition \eqref{eq:brinck-imag} reads:
\[
\left| \int_J d\widetilde{Q}_1(x) \right| \le C \qquad \forall J \subset \R_+, \; |J| \le 1.
\]
Then $\dot{f}[u] = \int_{\R_+} |u|^2 d\widetilde{Q}_1$. By \cite[Corollary 2.3.2]{MkhMlb2018}, for any $\epsilon \in (0, 1]$ and $C \ge 2$:
\[
|\dot{f}[u]| \le C\epsilon \int_{\R_+} |u'|^2 dx + 2C\epsilon^{-1} \int_{\R_+} |u|^2 dx.
\]
For any $k > 0$, choose $\epsilon := (2C(1 + k))^{-1} \in (0, 1]$. Then:
\begin{align*}
\dot{h}[u] - k|\dot{f}[u]| &\ge \int_{\R_+} |u'|^2 dx - C(1 + k)\epsilon \int_{\R_+} |u'|^2 dx - 2C(1 + k)\epsilon^{-1} \int_{\R_+} |u|^2 dx \\
&= \left( 1 - \frac{1}{2} \right) \int_{\R_+} |u'|^2 dx - 4C^2(1 + k)^2 \int_{\R_+} |u|^2 dx \\
&\ge -4C^2(1 + k)^2 \|u\|_{L_2}^2.
\end{align*}
Hence, $\dot{S}_0$ and $\dot{S}_0^+$ are sectorial with vertex $\gamma = -4C^2(1+k)^2$ and semi-angle $\theta = \arctan(1/k)$.

The proof is complete.

%%%%%%%%%%%%%%%%%%%%%%%%%%%%%%%%%%%%%%%%%%%%%%%%%%%%%%%%%%%%%%%%%%%%%%%%%%%%%%%%%%%%%%%%%%%%%%%%%%%%%%%%%%%%%%%%%%%%%%%%%%%%%%%%%%%%%%%%%%%%%%%%%%%%%
%%%%%%%%%%%%%%%%%%%%%%%%%%%%%%%%%%%%%%%%%%%%%%%%%%%%%%%%%%%%%%%%%%%%%%%%%%%%%%%%%%%%%%%%%%%%%%%%%%%%%%%%%%%%%%%%%%%%%%%%%%%%%%%%%%%%%%%%%%%%%%%%%%%%%
\section{Proof of Theorem~\ref{thm:3.2}}\label{sec:thm3.2}

Let the preminimal operators $\dot{S}_0$ and $\dot{S}_0^+$ be accretive. By applying a shift $S \mapsto S + \alpha I$ with $\alpha = 1$, we may assume without loss of generality that they are strictly accretive:
\begin{align}
\mathrm{Re}\,(\dot{S}_0 u, u)_{L_2(\R_+)} &\ge \|u\|_{L_2(\R_+)}^2 \qquad \forall u \in \Dom(\dot{S}_0), \label{eq:5.1}\\
\mathrm{Re}\,(\dot{S}_0^+ v, v)_{L_2(\R_+)} &\ge \|v\|_{L_2(\R_+)}^2 \qquad \forall v \in \Dom(\dot{S}_0^+). \label{eq:5.2}
\end{align}
Such a shift leaves the operator domains invariant and preserves all adjoint relations. 

By continuity, inequality \eqref{eq:5.1} extends to the entire domain of the closure $S_0 = \overline{\dot{S}_0}$:
\begin{equation*}\label{eq:S0-strict}
\mathrm{Re}\,(S_0 u, u)_{L_2(\R_+)} \ge \|u\|_{L_2(\R_+)}^2 \qquad \forall u \in \Dom(S_0).
\end{equation*}
From this estimate, it follows that the range $\Ran(S_0)$ is closed in $L_2(\R_+)$, and the defect index of $S_0$ is given by:
\begin{equation*}\label{eq:defect}
\mathrm{def}(S_0) := \dim \Ran(S_0)^\perp = \dim \Ker(S_0^*) = \dim \Ker(S^+).
\end{equation*}
Consequently, $S_0$ is $m$-accretive if and only if $\mathrm{def}(S_0) = 0$, that is, the homogeneous adjoint equation
\begin{equation}\label{eq:homog-adj}
S^+ v = 0, \qquad v \in \Dom(S^+)
\end{equation}
admits only the trivial solution $v = 0$.

To prove that $\Ker(S^+) = \{0\}$, we introduce a sequence of smooth cut-off functions. For each $n \in \N$, let $\Delta_n := [n, n+1]$. We choose functions $\varphi_n \in W_{2,\mathrm{comp}}^2(\R_+)$ satisfying the following properties:
\begin{enumerate}
    \item[(a1)] $0 \le \varphi_n(x) \le 1$ for all $x \in \R_+$;
    \item[(a2)] $\varphi_n(x) = 1$ for $x \in [0, n]$;
    \item[(a3)] $\supp \varphi_n \subset [0, n+1]$;
    \item[(a4)] $\varphi_n'(x) \le 0$ for $x \in \Delta_n$;
    \item[(a5)] there exists a constant $K > 0$, independent of $n$, such that $|\varphi_n'(x)| \le K$ for all $x \in \R_+$.
\end{enumerate}

\begin{lemmaEng}[Commutator Identity with Cut-Off Functions]\label{lem:cut-off}
For any cut-off function $\varphi_n$ ($n \in \N$) and any functions $u \in \Dom(S)$, $v \in \Dom(S^+)$, the products $\varphi_n u$ and $\varphi_n v$ satisfy:
\begin{enumerate}
    \item[(i)] $\varphi_n(0)u(0) = 0$, $\varphi_n(0)v(0) = 0$;
    \item[(ii)] $\varphi_n u \in \AC_{\mathrm{comp}}(\R_+)$, $\varphi_n v \in \AC_{\mathrm{comp}}(\R_+)$;
    \item[(iii)] $(\varphi_n u)^{[1]} = \varphi_n' u + \varphi_n u^{[1]} \in \AC_{\mathrm{comp}}(\R_+)$, $(\varphi_n v)^{\{1\}} = \varphi_n' v + \varphi_n v^{\{1\}} \in \AC_{\mathrm{comp}}(\R_+)$;
    \item[(iv)] Cancellation of Potentials in the Commutator:
    \begin{align}
    l[\varphi_n u] &= \varphi_n l[u] - \varphi_n'' u - 2\varphi_n' u' \in L_{2,\mathrm{comp}}(\R_+), \notag \\ %\label{eq:l-comm}
    l^+[\varphi_n v] &= \varphi_n l^+[v] - \varphi_n'' v - 2\varphi_n' v' \in L_{2,\mathrm{comp}}(\R_+); \label{eq:lplus-comm}
    \end{align}
    \item[(v)] $\varphi_n u \in \Dom(\dot{S}_0)$, $\varphi_n v \in \Dom(\dot{S}_0^+)$.
\end{enumerate}
\end{lemmaEng}

\begin{proof}[Proof of Lemma~\ref{lem:cut-off}]
\noindent(i) Follows from $u(0) = v(0) = 0$ and $\varphi_n(0) = 1$.

\noindent(ii) Since $\varphi_n \in W_{2,\mathrm{comp}}^2(\R_+) \subset C^1(\R_+)$ and $u, v \in \AC_{\mathrm{loc}}(\R_+)$, their products are absolutely continuous on compact sets. Compact support follows from $\supp \varphi_n \subset [0, n+1]$.

\noindent(iii) Compute the first quasi-derivative of $w := \varphi_n u$:
\[
w^{[1]} = w' - Qw = (\varphi_n' u + \varphi_n u') - Q\varphi_n u = \varphi_n' u + \varphi_n(u' - Qu) = \varphi_n' u + \varphi_n u^{[1]}.
\]
Since $\varphi_n, \varphi_n' \in \AC_{\mathrm{loc}}(\R_+)$ and $u, u^{[1]} \in \AC_{\mathrm{loc}}(\R_+)$, the product $(\varphi_n u)^{[1]}$ belongs to $\AC_{\mathrm{loc}}(\R_+)$. Compactness of support is preserved. The identical calculation yields $(\varphi_n v)^{\{1\}} = \varphi_n' v + \varphi_n v^{\{1\}} \in \AC_{\mathrm{comp}}(\R_+)$.

\noindent(iv) Apply the second quasi-derivative to $w = \varphi_n u$:
\[
l[\varphi_n u] = - \left[ (w^{[1]})' + Qw^{[1]} + (Q^2 - s)w \right].
\]
Differentiating $w^{[1]} = \varphi_n' u + \varphi_n u^{[1]}$:
\[
(w^{[1]})' = \varphi_n'' u + \varphi_n' u' + \varphi_n' u^{[1]} + \varphi_n(u^{[1]})'.
\]
Substituting this along with $Qw^{[1]} = \varphi_n' Qu + \varphi_n Qu^{[1]}$ and $(Q^2 - s)w = \varphi_n(Q^2 - s)u$:
\begin{align*}
l[\varphi_n u] &= -\left[ \varphi_n'' u + \varphi_n' u' + \varphi_n' u^{[1]} + \varphi_n(u^{[1]})' \right] - \left[ \varphi_n' Qu + \varphi_n Qu^{[1]} \right] - \varphi_n(Q^2 - s)u \\
&= \varphi_n \left[ -(u^{[1]})' - Qu^{[1]} - (Q^2 - s)u \right] - \varphi_n'' u - \varphi_n' \left( u' + u^{[1]} + Qu \right).
\end{align*}
By \eqref{eq:2.2-1}, $u^{[1]} + Qu = u'$. Therefore,
\[
u' + u^{[1]} + Qu = u' + u' = 2u'.
\]
Hence,
\[
l[\varphi_n u] = \varphi_n l[u] - \varphi_n'' u - 2\varphi_n' u'.
\]
The identical calculation for $l^+[\varphi_n v]$, using $v^{\{1\}} + \overline{Q}v = v'$, yields:
\[
l^+[\varphi_n v] = \varphi_n l^+[v] - \varphi_n'' v - 2\varphi_n' v'.
\]
Since $u \in \Dom(S)$, we have $u \in L_2(\R_+)$ and $u' = u^{[1]} + Qu \in L_{2,\mathrm{loc}}(\R_+)$. Since $\varphi_n'' \in L_2$ and $\varphi_n' \in L_\infty$ with support in $[n, n+1]$, we have $l[\varphi_n u] \in L_2(\R_+)$ and $l^+[\varphi_n v] \in L_2(\R_+)$.

\noindent(v) Combining (i)--(iv) and the compactness of supports shows that $\varphi_n u \in \Dom(\dot{S}_0)$ and $\varphi_n v \in \Dom(\dot{S}_0^+)$.
\end{proof}

\noindent\emph{Continuation of the Proof of Theorem~\ref{thm:3.2}.} 
Let $v \in \Dom(S^+)$ be a solution of the adjoint equation \eqref{eq:homog-adj}, i.e., $l^+[v] = 0$ and $v(0) = 0$. Define $v_n := \varphi_n v$. By Lemma~\ref{lem:cut-off}(v), $v_n \in \Dom(\dot{S}_0^+)$. Since $l^+[v] = 0$, formula \eqref{eq:lplus-comm} gives:
\[
l^+[v_n] = \varphi_n l^+[v] - \varphi_n'' v - 2\varphi_n' v' = -\varphi_n'' v - 2\varphi_n' v'.
\]
Compute the inner product $(\dot{S}_0^+ v_n, v_n)_{L_2} = (l^+[v_n], v_n)_{L_2}$:
\begin{equation}\label{eq:energy-1}
(l^+[v_n], v_n)_{L_2} = \int_0^\infty (-\varphi_n'' v - 2\varphi_n' v')\varphi_n \bar{v}\, dx = -\int_0^\infty \varphi_n'' \varphi_n |v|^2 dx - 2\int_0^\infty \varphi_n \varphi_n' v' \bar{v}\, dx.
\end{equation}
Integrating the first term by parts:
\[
-\int_0^\infty \varphi_n'' (\varphi_n |v|^2) dx = -\varphi_n' (\varphi_n |v|^2)\Big|_0^\infty + \int_0^\infty \varphi_n' (\varphi_n |v|^2)' dx.
\]
The boundary term at $x = 0$ vanishes because $v(0) = 0$, while at $x \to \infty$ it vanishes by the compact support of $\varphi_n$. Differentiating the product:
\[
(\varphi_n |v|^2)' = \varphi_n' |v|^2 + \varphi_n (|v|^2)' = \varphi_n' |v|^2 + \varphi_n (v'\bar{v} + v\bar{v}').
\]
Substituting this into the integral:
\[
-\int_0^\infty \varphi_n'' (\varphi_n |v|^2) dx = \int_0^\infty (\varphi_n')^2 |v|^2 dx + \int_0^\infty \varphi_n \varphi_n' (v'\bar{v} + v\bar{v}') dx.
\]
Plugging this into \eqref{eq:energy-1}:
\begin{align}
(l^+[v_n], v_n)_{L_2} &= \int_0^\infty (\varphi_n')^2 |v|^2 dx + \int_0^\infty \varphi_n \varphi_n' (v'\bar{v} + v\bar{v}') dx - 2\int_0^\infty \varphi_n \varphi_n' v' \bar{v}\, dx \notag\\
&= \int_0^\infty (\varphi_n')^2 |v|^2 dx + \int_0^\infty \varphi_n \varphi_n' (v\bar{v}' - v'\bar{v}) dx. \label{eq:energy-2}
\end{align}
Notice that $\varphi_n$ and $\varphi_n'$ are real-valued, while the term $v\bar{v}' - v'\bar{v} = 2i\,\mathrm{Im}(v\bar{v}')$ is purely imaginary. Therefore:
\[
\mathrm{Re}\left[ \varphi_n \varphi_n' (v\bar{v}' - v'\bar{v}) \right] = 0.
\]
Taking the real part of \eqref{eq:energy-2}, the second integral vanishes identically:
\begin{equation}\label{eq:energy-real}
\mathrm{Re}\,(\dot{S}_0^+ v_n, v_n)_{L_2(\R_+)} = \int_0^\infty (\varphi_n')^2 |v|^2 dx.
\end{equation}
On the other hand, by the strict accretivity assumption \eqref{eq:5.2} for $v_n \in \Dom(\dot{S}_0^+)$:
\[
\mathrm{Re}\,(\dot{S}_0^+ v_n, v_n)_{L_2} \ge \|v_n\|_{L_2}^2 = \int_0^\infty \varphi_n^2 |v|^2 dx \ge \int_0^n |v(x)|^2 dx,
\]
since $\varphi_n(x) = 1$ on $[0, n]$. Since $\supp \varphi_n' \subset [n, n+1]$ and $|\varphi_n'(x)| \le K$, equation \eqref{eq:energy-real} implies:
\begin{equation}\label{eq:tail-bound}
\int_0^n |v(x)|^2 dx \le \int_n^{n+1} (\varphi_n')^2 |v|^2 dx \le K^2 \int_n^{n+1} |v(x)|^2 dx.
\end{equation}
Since $v \in L_2(\R_+)$, the tail of the integral tends to zero as $n \to \infty$:
\[
\lim_{n \to \infty} \int_n^{n+1} |v(x)|^2 dx = 0.
\]
Passing to the limit as $n \to \infty$ in \eqref{eq:tail-bound}:
\[
\int_0^\infty |v(x)|^2 dx = \lim_{n \to \infty} \int_0^n |v(x)|^2 dx \le K^2 \lim_{n \to \infty} \int_n^{n+1} |v(x)|^2 dx = 0.
\]
Therefore, $v(x) = 0$ almost everywhere on $\R_+$.

Thus, $\Ker(S^+) = \{0\}$, which proves that $\mathrm{def}(S_0) = 0$, establishing the $m$-accretivity of $S_0$. Repeating the argument for the equation $Su = 0$ with $u \in \Dom(S)$ shows that $\Ker(S) = \{0\}$, which gives $\mathrm{def}(S_0^+) = 0$ and the $m$-accretivity of $S_0^+$.

\smallskip
\noindent\emph{Proof of the Coincidence $S_0 = S$ and $S_0^+ = S^+$.} 
Since $S_0$ is $m$-accretive, its adjoint $S_0^* = S^+$ is also $m$-accretive. On the other hand, the minimal operator $S_0^+$ is $m$-accretive, and by construction $S_0^+ \subset S^+$. Since an $m$-accretive operator admits no proper accretive extensions in $H$, the inclusion $S_0^+ \subset S^+$ of two $m$-accretive operators immediately forces:
\[
S_0^+ = S^+.
\]
Taking adjoints on both sides yields $(S_0^+)^* = (S^+)^*$, which by Theorem~\ref{thm:2.1}($2^\circ$) gives:
\[
S = S_0.
\]
Finally, from the domain characterization \eqref{eq:DomS0}, $\Dom(S_0) = \Dom(S)$ is equivalent to $[u, v](\infty)$ $= 0$ for all $u \in \Dom(S)$ and $v \in \Dom(S^+)$, establishing \eqref{eq:vanish-boundary}. 

The proof is complete.

%%%%%%%%%%%%%%%%%%%%%%%%%%%%%%%%%%%%%%%%%%%%%%%%%%%%%%%%%%%%%%%%%%%%%%%%%%%%%%%%%%%%%%%%%%%%%%%%%%%%%%%%%%%%%%%%%%%%%%%%%%%%%%%%%%%%%%%%%%%%%%%%%%%%%
%%%%%%%%%%%%%%%%%%%%%%%%%%%%%%%%%%%%%%%%%%%%%%%%%%%%%%%%%%%%%%%%%%%%%%%%%%%%%%%%%%%%%%%%%%%%%%%%%%%%%%%%%%%%%%%%%%%%%%%%%%%%%%%%%%%%%%%%%%%%%%%%%%%%%

%%%%%%%%%%%%%%%%%%%%%%%%%%%%%%%%%%%%%%%%%%%%%%%%%%%%%%%%%%%%%%%%%%%%%%%%%%%%%%%%%%%%%%%%%%%%%%%%%%%%%%%%%%%%%%%%%%%%%%%%%%%%%%%%%%%%%%%%%%%%%%%%%%%%%
%%%%%%%%%%%%%%%%%%%%%%%%%%%%%%%%%%%%%%%%%%%%%%%%%%%%%%%%%%%%%%%%%%%%%%%%%%%%%%%%%%%%%%%%%%%%%%%%%%%%%%%%%%%%%%%%%%%%%%%%%%%%%%%%%%%%%%%%%%%%%%%%%%%%%

%%%%%%%%%%%%%%%%%%%%%%%%%%%%%%%%%%%%%%%%%%%%%%%%%%%%%%%%%%%%%%%%%%%%%%%%%%%%%%%%%%%%%%%%%%%%%%%%%%%%%%%%%%%%%%%%%%%%%%%%%%%%%%%%%%%%%%%%%%%%%%%%%%%%%
%%%%%%%%%%%%%%%%%%%%%%%%%%%%%%%%%%%%%%%%%%%%%%%%%%%%%%%%%%%%%%%%%%%%%%%%%%%%%%%%%%%%%%%%%%%%%%%%%%%%%%%%%%%%%%%%%%%%%%%%%%%%%%%%%%%%%%%%%%%%%%%%%%%%%
\section*{Funding}
The first named author would like to thank the Isaac Newton Institute for Mathematical Sciences, Cambridge, for support and hospitality during the "Solidarity Program" where work was supported by the "EPSRC grant no EP/RO14604/1".

The author wishes to thank the Department of Mathematics, King's College London, for their hospitality, and to the Ministry of Education and Science of Ukraine for support under the grant 0126U000898.

%%%%%%%%%%%%%%%%%%%%%%%%%%%%%%%%%%%%%%%%%%%%%%%%%%%%%%%%%%%%%%%%%%%%%%%%%%%%%%%%%%%%%%%%%%%%%%%%%%%%%%%%%%%%%%%%%%%%%%%%%%%%%%%%%%%%%%%%%%%%%%%%%%%%%
%%%%%%%%%%%%%%%%%%%%%%%%%%%%%%%%%%%%%%%%%%%%%%%%%%%%%%%%%%%%%%%%%%%%%%%%%%%%%%%%%%%%%%%%%%%%%%%%%%%%%%%%%%%%%%%%%%%%%%%%%%%%%%%%%%%%%%%%%%%%%%%%%%%%%


\begin{thebibliography}{99}

\bibitem{AlbGszHKH2005}
 {Albeverio, S.; Gesztesy, F.; H{ø}egh Krohn, R.; Holden, H.}
 {\textit{Solvable models in quantum mechanics}},
 {2nd ed.},
 {With an appendix by Pavel Exner};
 {AMS Chelsea Publishing}:
 {Providence, RI},
 {2005};
 {xiv+488~pp.};
 doi:~{10.1090/chel/350}.


\bibitem{AlKsMl2010}
 {Albeverio, S.; Kostenko, A.; Malamud, M.}
 {Spectral theory of semibounded {S}turm-{L}iouville operators with local interactions on a discrete set}.
 {\textit{J. Math. Phys.}}
 {\textbf{2010}},
 {51},
 {no.~10},
 {102102},
 {24~pp.};
 {doi:~{10.1063/1.3490672}.}


\bibitem{ArlSchm2026}
 {Arlinskii, Y.; Schm\"udgen, K.}
 {On {$C$}-symmetric and {$C$}-self-adjoint unbounded operators on {H}ilbert space}.
 {\emph{Integr. Equ. Oper. Theory}}
 {\textbf{2026}}
 {98},
 {n.~2},
 {Paper N.~14, 25},
 {doi:~{10.1007/s00020-026-02831-3}}.
 
 
 

\bibitem{Brsch1985}
 {Brasche, J.},
 {\textit{Perturbation of Schr\"{o}dinger Hamiltonians by measures --- selfadjointness and semiboundedness}},
 {J. Math. Phys.}
 {\textbf{26}}
 {(1985)},
 {no.~4},
 (621--626).


\bibitem{Brnk1959}
 {Brinck, I.},
 {\textit{Self-adjointness and spectra of Sturm-Liouville operators}},
 {Math. Scand.},
 {\textbf{7}}
 {(1959)},
 {no.~1},
 {219--239}.


%\bibitem{ClGs2003}
% {Clark, S.; Gesztesy, F.}
% {On Povzner--Weinholtz-type self-adjointness results for matrix-valued Sturm--Liouville operators}.
% {\textit{Proceedings of the Royal Society of Edinburgh}}
% {\textbf{2003}},
% {133A},
% {no.~4},
% {747--758};
% doi:~{10.1017/S0308210500002651}.


%\bibitem{DnfSchw_II_1963}
% {Dunford, N.; Schwartz, J. T.}
% {\textit{Linear operators. Part II: Spectral theory. Selfadjoint operators in Hilbert space}}.
% {With the assistance of William G. Bade and Robert G. Bartle.};
% {Interscience Publishers John Wiley \& Sons}:
% {New York - London},
% {1963};
% {ix+pp. 859--1923+7};
% {\MR{0188745}}.

\bibitem{EcGsNcTs2013}
 {Eckhardt, J.; Gesztesy, F.; Nichols, R.; Teschl, G.}
 {Weyl--Titchmarsh theory for Sturm--Liouville operators with distributional potentials}.
 {\textit{Opuscula Math.}}
 {\textbf{2013}},
 {33},
 {n.~3},
 {467--563};
 doi:~{10.7494/OpMath.2013.33.3.467}.


\bibitem{EvMrksBook1999}
 {Everitt, W. N.; Markus, L.}
 {\textit{Boundary value problems and symplectic algebra for ordinary differential and quasi-differential operators}}.
 {Mathematical Surveys and Monographs},
 {Vol.~\textbf{61}},
 {American Mathematical Society, Providence, RI},
 {1999};
 {xii+187};
 {{MR1647856}}.


\bibitem{Glzmn1966}
 {Glazman, I.}
 {\textit{Direct methods of qualitative spectral analysis of singular differential operators}};
 {Daniel Davey \& Co., Inc.}:
 {New York},
 {1966};
 {ix+234~pp.};
 {MR{0190800}}.


\bibitem{GszNclZnch2024}
 {Gesztesy, F.; Nichols, R.; Zinchenko, M.}
 {\textit{Sturm-{L}iouville operators, their spectral theory, and some  applications}};
 {American Mathematical Society Colloquium Publications},
 {\textbf{67}},
 {American Mathematical Society}:
 {Providence, RI},
 {2024};
 {xv+927~pp.};
 doi:~{10.1090/coll/067}.


\bibitem{GrnMkh2010}
 {Goriunov, A.; Mikhailets, V.}
 {Regularization of singular Sturm-Liouville equations}.
 {\textit{Methods Funct. Anal. Topology}}
 {\textbf{2010}},
 {16},
 {no.~2},
 {120--130};
 {MR{2667807}}.


\bibitem{GrnMkhPnk2013}
 {Goriunov, A.; Mikhailets, V.; Pankrashkin, K.}
 {Formally self-adjoint quasi-differential operators and boundary-value problems}.
 {\textit{Electron. J. Differential Equations}}
 {\textbf{2013}},
 {No.~101},
 {16~pp.};
 {MR{3065054}}.


%\bibitem{Hrtmn1948}
% {Hartman, P.},
% {Differential equations with non-oscillatory eigenfunctions}.
% {\textit{Duke Math. J.}}
% {\textbf{1948}},
% {15},
% {697--709};
% {\MR{0027927}}.


\bibitem{HrMk2012}
 {Hryniv, R.; Mykytyuk, Ya.}
 {Self-adjointness of {S}chr\"odinger operators with singular potentials}.
 {\textit{Methods Funct. Anal. Topology}}
 {\textbf{2012}},
 {18},
 {no.~2},
 {152--159};
 {MR{2978191}}.


%\bibitem{Ismgl1963}
% {Ismagilov, R. S.}
% {On the self-adjointness of the Sturm-Liouville operator}
% {(Russian)}.
% {\textit{Uspehi Mat. Nauk}}
% {\textbf{1963}},
% {18},
% {no.~5(113)},
% {161--166};
% {MR{0155037}}.


\bibitem{Kato1995}
 {Kato, T.}
 {\textit{Perturbation theory for linear operators}};
 {Classics in Mathematics};
 {Springer-Verlag}:
 {Berlin},
 {1995};
 {xxii+619~pp.};
 {MR{1335452}}.


%\bibitem{KsMlNc2022}
% {Kostenko, A.; Malamud, M.; Nicolussi, N.}
% {A {G}lazman-{P}ovzner-{W}ienholtz theorem on graphs}.
% {\textit{Adv. Math.}}
% {\textbf{2022}},
% {395},
% {Paper No.~108158},
% {30~pp.};
% doi:~{10.1016/j.aim.2021.108158}.


%\bibitem{Ldsk1960}
% {Lidski\u i, V.},
% {A non-selfadjoint operator of {S}turm-{L}iouville type with a discrete spectrum},
% {\textit{Trudy Moskov. Mat. Ob\v{s}\v{c}.}}
% {\textbf{1960}},
% {9},
% {45--79};
% {\MR146442}.


%\bibitem{MzVrbActaMathSin2019}
% {Maz'ya, V.; Verbitsky, I. E.}
% {Accretivity of the general second order linear differential operator}.
% {\textit{Acta Math. Sin. (Engl. Ser.)}}
% {\textbf{2019}},
% {35},
% {no.~6},
% {832--852};
% doi:~{10.1007/s10114-019-8127-9}.


\bibitem{MkhGrnMlb2022}
 {Mikhailets, V.; Goriunov, A.; Molyboga, V.}
 {Povzner--Wienholtz-type theorems for Sturm--Liouville operators with singular coefficients}.
 {\textit{Complex Anal. Oper. Theory}}
 {\textbf{2022}},
 {16},
 {no.~8},
 {Paper No.~113},
 {13~pp.};
 doi:~{10.1007/s11785-022-01291-y}.


\bibitem{MkhMlb2013_mfat19n1}
 {Mikhailets, V.; Molyboga, V.}
 {Schr\"odinger operators with complex singular potentials}.
 {\textit{Methods Funct. Anal. Topology}}
 {\textbf{2013}},
 {19},
 {no.~1},
 {16--28};
 {MR{3088075}}.


\bibitem{MkhMlb2013_mfat19n2}
 {Mikhailets, V.; Molyboga, V.}
 {Remarks on Schr\"odinger operators with singular matrix potentials}.
 {\textit{Methods Funct. Anal. Topology}}
 {\textbf{2013}},
 {19},
 {no.~2},
 {161--167};
 {MR{3098494}}.


\bibitem{MkhMlb2018}
 {Mikhailets, V.; Molyboga, V.}
 {Schr\"odinger operators with measure-valued potentials: semiboundedness and spectrum}.
 {\textit{Methods Funct. Anal. Topology}}
 {\textbf{2018}},
 {24},
 {no.~3},
 {240--254};
 {MR{3860804}}.

\bibitem{MkhMlb2026}
 {Mikhailets, V.; Molyboga, V.}
 {Uniqueness problem for accretive Schr\"{o}dinger operators with complex singular coefficients}.
 {\emph{Bull. London Math. Soc.}}
 {\textbf{2026}}
 {58},
 {9},
 {1--27};
 {doi:~10.1112/blms.70490}.


%\bibitem{MkhMrchNvk2017}
% {Mikhailets, V.; Murach, A.; Novikov, V.}
% {Localization principles for Schr\"odinger operator with a singular matrix potential}.
% {\textit{Methods Funct. Anal. Topology}}
% {\textbf{2017}},
% {23},
% {no.~4},
% {367--377};
% {\MR{3745187}}.


%\bibitem{MkhSbl1999}
% {Mikhailets, V.; Sobolev, A.}
% {Common eigenvalue problem and periodic Schr\"odinger operators}.
% {\textit{J. Funct. Anal.}}
% {\textbf{1999}},
% {165},
% {no.~1},
% {150--172};
% doi:~{10.1006/jfan.1999.3406}.


%\bibitem{Mlch1953}
% {Mol\v{c}anov, A.},
% {On conditions for discreteness of the spectrum of self-adjoint differential equations of the second order},
% {\textit{Trudy Moskov. Mat. Ob\v{s}\v{c}.}}
% {\textbf{1953}},
% {2},
% {169--199};
% {\MR{57422}}.


%\bibitem{NmrkMMO1954}
% {Na\u imark, M.},
% {Investigation of the spectrum and the expansion in eigenfunctions of a nonselfadjoint operator of the second order on a %semi-axis},
% {\textit{Trudy Moskov. Mat. Ob\v{s}\v{c}.}}
% {\textbf{1954}},
% {3},
% {181--270};
% {\MR{62311}}.


%\bibitem{Naimark}
% {M.~Naimark},
% {\textit{Linear differential operators. Part II: Linear differential operators in Hilbert space}}.
% {With additional material by the author, and a supplement by V. È. Ljance}.
% {Translated from the Russian by E. R. Dawson. English translation edited by W. N. Everitt}.
% {Frederick Ungar Publishing Co.},
% {New York},
% {1968},
% {xv+352 pp.}


\bibitem{NmrkRus1969}
 {Naimark, M.}
 {\textit{Linear Differential Operators}}
 {(Russian)},
 {2nd ed.};
 {With an appendix by V. \`E. Ljance};
 {Nauka}:
 {Moscow},
 {1969};
 {526~pp.};
 {MR{0353061}}.


%\bibitem{Orchk1988}
% {Orochko, Yu.\ B.}
% {The hyperbolic equation method in the theory of operators of {S}chr\"odinger type with locally integrable potential}.
% {\textit{Uspekhi Mat. Nauk}}
% {\textbf{1988}},
% {43},
% {no.~2(260)},
% {43--86, 230};
% doi: {10.1070/RM1988v043n02ABEH001728}.


%\bibitem{Orchk1974}
% {Oro\v cko, Ju.\ B.}
% {Sufficient conditions for the selfadjointness of a {S}turm-{L}iouville operator}
% {(Russian)}.
% {\textit{Mat. Zametki}}
% {\textbf{1974}},
% {15},
% {271--280};
% {\MR{0375002}}.


\bibitem{Phl1959}
 {Phillips, R. S.}
 {Dissipative operators and hyperbolic systems of partial differential equations}.
 {\textit{Trans. Amer. Math. Soc.}}
 {\textbf{1959}},
 {90},
 {193--254};
 doi:~{10.2307/1993202}.


%\bibitem{Pvznr1953}
% {Povzner, A. Ya.}
% {On the expansion of arbitrary functions in eigenfunctions of the operator $-\Delta u+cu$}
% {(Russian)}.
% {\textit{Mat. Sbornik N.S.}}
% {\textbf{1953}},
% {32/74},
% {109--156};
% {\MR{0053330}}.


%\bibitem{ReedSimon-book02_Eng_1975}
% {Reed, M.; Simon, B.}
% {\textit{Methods of modern mathematical physics. II. Fourier analysis, self-adjointness}};
% {Academic Press [Harcourt Brace Jovanovich, Publishers]}:
% {New York - London},
% {1975};
% {xv+361~pp.};
% {\MR{493420}}.


%\bibitem{ReedSimon-book04_Eng_1978}
% {Reed, M.; Simon, B.}
% {\textit{Methods of modern mathematical physics. IV. {A}nalysis of operators}};
% {Academic Press [Harcourt Brace Jovanovich, Publishers]}:
% {New York - London},
% {1978};
% {xv+396~pp.};
% {\MR{493421}}.


%\bibitem{Rllch1951}
% {Rellich, F.}
% {Halbbeschr\"ankte gew\"ohnliche Differentialoperatoren zweiter Ordnung}.
% {\textit{Math. Ann.}}
% {\textbf{1951}},
% {122},
% {343--368};
% doi:~{10.1007/BF01342848}.

%\bibitem{Sears1950}
% {Sears, D. B.}
% {Note on the uniqueness of the Green's functions associated with certain differential equations}.
% {\textit{Canad. J. Math.}}
% {\textbf{1950}},
% {2},
% {314--325};
% doi:~{10.4153/cjm-1950-029-9}.

\bibitem{Schm2012}
 {Schm\"{u}dgen, K.}
 {\textit{Unbounded self-adjoint operators on Hilbert space}};
 {Graduate Texts in Mathematics},
 {265};
 {Springer}:
 {Dordrecht},
 {2012};
 {xx+432~pp.};
 doi:~{10.1007/978-94-007-4753-1}.


\bibitem{Shin1938_I}
 {D.~Shin},
 {\textit{Existence theorem for quasi-differential equation of $n$-th order}}
 {(Russian)},
 {Dokl. Akad. Nauk SSSR}
 {\textbf{XIII}}
 {(1938)},
 {no.~8},
 {515--518}.

\bibitem{Shin1938_II}
 {D.~Shin},
 {\textit{About solutions of the self-adjoint differential equation $u^{[n]}=l u$, $\textbf{I}(l)\neq 0$, belonging to $L_{2}(0,\infty)$}}
 {(Russian)},
 {Dokl. Akad. Nauk SSSR}
 {\textbf{XIII}}
 {(1938)},
 {no.~8},
 {519--522}.

\bibitem{Shin1938_III}
 {D.~Shin},
 {\textit{About quasi-differential operators in Hilbert space}}
 {(Russian)},
 {Dokl. Akad. Nauk SSSR}
 {\textbf{XIII}}
 {(1938)},
 {no.~8},
 {523--516}.

\bibitem{Shin1940}
 {D.~Shin},
 {\textit{About solutions of quasi-differential equation of $n$-th order}}
 {(Russian)},
 {Matem. Zb.}
 {\textbf{7}}
 {(\textbf{49})}
 {(1940)},
 {479--532}.
 {MR{0002428}}.


\bibitem{Shin1943}
 {D.~Shin},
 {\textit{About quasi-differential operators in a Hilbert space}}
 {(Russian)},
 {Matem. Zb.}
 {\textbf{13}}
 {(\textbf{55})}
 {(1943)},
 {39--70};
 {MR{0011534}}.


%\bibitem{Schmincke1991}
% {Schmincke, U.-W.}
% {Proofs of Povzner-Wienholtz type theorems on selfadjointness of  Schr\"odinger operators by means of positive 
%eigensolutions},
% {\textit{Bull. London Math. Soc.}}
% {\textbf{1991}},
% {23},
% {no.~3},
% {263--266};
% doi:~{10.1112/blms/23.3.263}.


%\bibitem{Shbn2001}
% {Shubin, M.}
% {Essential self-adjointness for semi-bounded magnetic {S}chr\"odinger operators on non-compact manifolds}.
% {\textit{J. Funct. Anal.}}
% {\textbf{2001}},
% {186},
% {no.~1},
% {92--116};
% doi:~{10.1006/jfan.2001.3778}.


%\bibitem{Smdr1978}
% {Simader, C. G.}
% {Essential self-adjointness of Schr\"odinger operators bounded from below}.
% {\textit{Math. Z.}}
% {\textbf{1978}},
% {159},
% {no.!1},
% {47--50};
% doi:~{10.1007/BF01174567}.


%\bibitem{Smdr1992}
% {Simader, C. G.}
% {Remarks on essential selfadjointness of Schr\"odinger operators with singular electrostatic potentials}.
% {\textit{J. Reine Angew. Math.}}
% {\textbf{1992}},
% {431},
% {1--6};
% doi:~{10.1515/crll.1992.431.1}.


%\bibitem{StHnsn1966}
% {Stetkaer-Hansen, H.}
% {A generalization of a theorem of Wienholtz concerning essential selfadjointness of singular elliptic operators}.
% {\textit{Math. Scand.}}
% {\textbf{1966}},
% {19},
% {108--112};
% doi:~{10.7146/math.scand.a-10798}.


\bibitem{Wei1987}
 {Weidmann, J.}
 {\textit{Spectral theory of ordinary differential operators}};
 {Lecture Notes in Mathematics},
 {1258},
 {Springer-Verlag}:
 {Berlin},
 {1987};
 {vi+303~pp.};
 doi:~{10.1007/BFb0077960}.

%\bibitem{Weyl1910}
% {Weyl, H.}
% {\"Uber gew\"ohnliche Differentialgleichungen mit Singularit\"aten und die zugeh\"origen Entwicklungen willk\"urlicher 
% Funktionen}
% {(German)}.
% {\textit{Math. Ann.}}
% {\textbf{1910}},
% {68},
% {no.~2},
% {220--269};
% doi:~{10.1007/BF01474161}.

%\bibitem{Wnhltz1958}
% {Wienholtz, E.}
% {Halbbeschr\"ankte partielle Differentialoperatoren zweiter Ordnung vom elliptischen Typus}.
% {\textit{Math. Ann.}}
% {\textbf{1958}},
% {135},
% {50--80};
% doi:~{10.1007/BF01350827}.


\bibitem{Zttl1975}
 {Zettl, A.}
 {Formally self-adjoint quasi-differential operator}.
 {Rocky Mountain J. Math.}
 {\textbf{1975}},
 {5},
 {no.~3},
 {453--474};
 doi:~{10.1216/RMJ-1975-5-3-453}.


%\bibitem{Zttl2005}
% {Zettl, A.}
% {\textit{Sturm--Liouville Theory}};
% {Mathematical Surveys and Monographs},
% {121};
% {American Mathematical Society}:
% {Providence, RI},
% {2005};
% {xii+328~pp.};
% doi:~{10.1090/surv/121}.


%\bibitem{Zttl2021}
% {Zettl, A.}
% {\textit{Recent developments in {S}turm-{L}iouville theory}};
% {De Gruyter Studies in Mathematics},
% {76},
% {De Gruyter}:
% {Berlin},
% {2021};
% {xii+245~pp.};
% doi:~{10.1515/9783110719000}.


\end{thebibliography}
\end{document}